\documentclass{amsart}

\usepackage{amsmath,amssymb,amscd} \usepackage{graphicx}
\DeclareGraphicsExtensions{.eps} \usepackage{mathrsfs}
\usepackage[mathcal]{eucal}

\newtheorem{Thm}{Theorem}{\bfseries}{\itshape}
\newtheorem{Cor}{Corollary}{\bfseries}{\itshape}
\newtheorem{Prop}[Cor]{Proposition}{\bfseries}{\itshape}
\newtheorem{Lem}[Cor]{Lemma}{\bfseries}{\itshape}
{\bfseries}{\itshape}
\newtheorem{Def}[Cor]{Definition}{\bfseries}{\rmfamily}
{\scshape}{\rmfamily}
{\scshape}{\rmfamily}

 \renewcommand\le{\leqslant}
\let\tildeaccent=\~ \let\hataccent=\^
\renewcommand\~[1]{\widetilde{#1}} \renewcommand\^[1]{\widehat{#1}}

\def\<{\left<} \def\>{\right>} \def\({\left(} \def\){\right)}

\let\subsetneq=\subsetneqq 

\let\polishL=l \def\Zoladek.{\.Zol\c adek}

 \def\etc.{\emph{etc}.}
 
\def\:{\colon}  \def\C{{\mathbb C}}

 \def\d{\,\mathrm d}

 \let\PolishL=\L \def\Lojas.{\PolishL ojasiewicz}
\def\L{\varLambda}  
  
\def\cF{{\mathcal F}} \def\cL{{\mathcal L}} 
\def\cI{{\mathcal I}} \def\cJ{{\mathcal J}} \def\cK{{\mathcal K}}

\def\cO{{\mathcal O}}

 \def\mult{\operatorname{mult}}

\def\V{V} \def\trans{\pitchfork}

\def\rest#1{{\vert_{#1}}} \def\onL{\rest{\cL}}

\def\LT{\operatorname{LT}}

\begin{document}

\title{Intersection multiplicities of Noetherian functions}
\author{Gal Binyamini, Dmitry Novikov} \address{Weizmann Institute of
  Science\\Rehovot\\Israel}

\begin{abstract}We provide a partial answer to the following problem:
  \emph{give an effective upper bound on the multiplicity of
    non-isolated common zero of a tuple of Noetherian functions}. More
  precisely, consider a foliation defined by two commuting polynomial
  vector fields $V_1,V_2$ in $\C^n$, and $p$ a nonsingular point of
  the foliation. Denote by $\cL$ the leaf passing through $p$, and let
  $F,G\in\C[X]$ be two polynomials. Assume that $F\onL=0,G\onL=0$ have
  several common branches. We provide an effective procedure which
  allows to bound from above multiplicity of intersection of
  remaining branches of $F\onL=0$ with $G\onL=0$ in terms of the
  degrees and dimensions only.
\end{abstract}
\maketitle
\date{\today}

\section{Introduction}
Let $\cF$ be the foliation generated in $X=\C^n$ by several commuting
polynomial vector fields $V_1,..., V_k$. A restriction of a polynomial
$F\in \C[X]$ to a leaf of $\cF$ is called a \emph{Noetherian
  function}.  The main result of this paper is motivated by the
following question: \emph{is it possible to effectively bound the
  topological complexity of objects defined by Noetherian functions
  solely in terms of discrete parameters of the defining functions
  (i.e. dimension of spaces, number and degrees of vector fields and
  of polynomials)}?

An important subclass of the class of Noetherian functions is that of
Pfaffian functions.  The theory of Fewnomials developed by Khovanskii
provides effective upper bounds for global topological invariants
(e.g. Betti numbers) of \emph{real} varieties defined by Pfaffian
equations (see~\cite{Kho:Few}).  Evidently, for complex varieties such
bounds are impossible: the only holomorphic functions which admit
``finite complexity'' (e.g., finitely many zeros) in the entire
complex domain are the polynomials.  One may trace the dichotomy
between the real and complex settings to the absence of a complex
analogue of the Rolle theorem (which is a cornerstone of the real
Fewnomial theory) --- the derivative of a holomorphic function with
many zeros may have no zeros at all.

While the global Rolle theorem fails in the complex setting, certain
local analogs still hold. Perhaps the simplest of these analogs is the
statement that if a derivative $f'$ admits a zero of multiplicity $n$
at some point, then $f$ may admit a zero of multiplicity at most $n+1$
at the same point.  This trivial claim, and ramifications thereof, can
be used to build local theory of complex Pfaffian sets, and to provide
effective estimates on the \emph{local} complexity of complex Pfaffian
sets \cite{Gab:Loj}.

The topology of global real Noetherian sets is usually infinite, as
demonstrated by the simple Noetherian (but non-Pfaffian) function
$\sin(x)$, which admits infinitely many zeros.  However, a long
standing conjecture due to Khovanskii claims that the complexity of
the local topology of such functions can be estimated through the
discrete parameters of the set. This conjecture is motivated as
follows: in view of Morse theory, to estimate local Betti numbers it
is essentially sufficient to bound the number of critical points of
Noetherian functions on germs of Noetherian sets. The latter can be
bounded through a suitably defined multiplicity of a common zero of a
suitable tuple of Noetherian functions. In \cite{GabKho} the
multiplicity of an \emph{isolated} common zero is bounded from above.
To build the general theory one has to generalize this result to
non-isolated intersections.

For non-isolated intersections even the notion of multiplicity becomes
non-trivial. For a point $p$ lying on a leaf $\cL_0$ of the foliation
$\cF$ one can define the multiplicity of the common zero $p$ of the
functions $F_i|_{\cL_0}$ as the number of common isolated zeros of
$F_i|_{\cL}$ on neighboring leaves $\cL$ converging to $p$ as
$\cL\to\cL_0$. The multiplicity defined in this manner is not
intrinsic to the leaf $\cL$. It depends on the foliation $\cF$ in
which $\cL$ is embedded.

In this paper we restrict attention to the case of foliations with
two-dimensional leaves. In this context, we consider another notion of
non-isolated multiplicity, suggested by Gabrielov, which is defined
intrinsically on the leaf being considered (see
subsection~\ref{subsec:non-isolated-mult}).  We prove that this
multiplicity can be explicitly bounded in terms of the dimension $n$
and the degrees of the vector fields defining the foliation.

We would like to thank A. Gabrielov, A. Khovanskii and P. Milman for many fruitful discussions. Our special thanks go to S. Yakovenko for continuous help and advice from the very beginning.


\section{Setup and notation}

Let $\cF$ be the foliation generated in $X=\C^n$ by two commuting
polynomial vector fields $V_1, V_2$. Given two polynomial functions
$F,G$, we denote by $\{F,G\}$ the Poisson bracket of $F$ and $G$ with
respect to the leafs of the foliation,
\begin{equation}
  \{f,g\} = V_1(f) V_2(g)-V_2(f) V_1(g)
\end{equation}

In the entire paper $\cL$ denotes some particular fixed leaf of the
foliation, and $p\in\cL$ a particular fixed smooth point of $\cL$.  We will denote
functions defined on $X$ using capital letters, and functions defined
on $\cL$ by small letters. We will also denote ideals of functions on
$X$ by capital letter, and ideals of functions on $\cL$ by
calligraphic letters.

We denote by $\cO(X)$ the ring of polynomial functions on $X$, and by
$\cO_p(\cL)$ the ring of germs of analytic functions on $\cL$ at $p$.
Given an ideal $\cI\subset\cO_p(\cL)$ we denote by $\mult_p \cI$ the
dimension $\dim_\C \cO_p(\cL)/\cI$.  We also denote by $\V(I)$ the
variety associated to $I$.

\section{Noetherian pairs and controllable inclusions}

In studying the restriction of algebraic functions to the leafs of a
foliation we shall often find it necessary to simultaneously keep
track of functions defined locally on a particular leaf, and their
algebraic counterparts defined globally. In this section we introduce
notation and terminology to facilitate the manipulation of such data.

\begin{Def}
  A pair of ideals $(I,\cI)$ with
  $I\subset \cO(X), \cI\subset\cO_p(\cL)$ is called a \emph{noetherian
    pair} for the leaf $\cL$ if $I\onL \subset\cI$.
\end{Def}

By an inclusion of pairs $(I,\cI)\subset(J,\cJ)$ we mean simply that
$I\subset J,\cI\subset\cJ$. The pair $(J,\cJ)$ is said to
\emph{extend} the pair $(I,\cI)$.

We will introduce a number of operations which generate for a given
pair $(I,\cI)$ an extension $(J,\cJ)$. The goal is to form an
extension in such a way that the multiplicity of $\cI$ can be
estimated from that of $\cJ$. More precisely we introduce the
following notion.

\begin{Def}
  An inclusion $(I,\cI)\subset(J,\cJ)$ is said to be
  \emph{controllable} if:
  \begin{itemize}
  \item The complexity of $J$ can be bounded in terms of the
    complexity of $I$.
  \item The multiplicity of $\cI$ can be bounded in terms of the
    multiplicity of $\cJ$.
  \end{itemize}
\end{Def}

The precise estimates on the complexity and the multiplicity in the
definition above will vary for the various types of inclusions we
form.

We record two immediate consequences of this definition.

\begin{Prop} \label{prop:inclusion-sequence} If each inclusion in the
  sequence
  \begin{equation}
    (I_0,\cI_0)\subset\cdots\subset(I_k,\cI_k)
  \end{equation}
  is controllable, then the inclusion $(I_0,\cI_0)\subset(I_k,\cI_k)$
  is controllable.
\end{Prop}

\begin{Prop} \label{prop:inclusion-mult-bound} Suppose
  $(I,\cI)\subset(J,\cJ)$ is an controllable inclusion, and
  $p\not\in V(J)$. Then one can give an upper bound for $\mult_p\cI$.
\end{Prop}
\begin{proof}
  By definition, one can give an upper bound for $\mult_p(\cI)$ in
  terms of $\mult_p(\cJ)$. But $J\subset\cJ$ and $p\not\in V(J)$, so
  $\mult_p\cJ=0$ and the proposition follows.
\end{proof}

These two proposition lay out the general philosophy of this paper.
We start with a pair $(I,\cI)$ and attempt, by forming a sequence of
controllable inclusions, to reach a pair $(J,\cJ)$ with $V(J)$ as
small as possible. If we manage to get a pair with $p\not\in V(J)$
then we obtain an upper bound for $\mult_p\cI$.

We now introduce the three controllable inclusions which will be used
in this paper.

\subsection{Radical extension}

Given a pair $(I,\cI)$ we define a new pair $(J,\cJ)$ by letting
$J=\sqrt{I}$ and $\cJ=\<\cI,J\onL\>$. The complexity of the ideal $J$
can be estimated from that of $I$ using effective radical extraction
algorithms. The multiplicity of $\cI$ can be estimated from that of
$\cJ$ by combining the effective Nullstellensatz with the following
simple lemma.

\begin{Lem}
  Let $\cK\subset\cO_p(\cL)$ be an ideal of finite multiplicity, and
  suppose $f^n\in\cK$.  Then
  \begin{equation}
    \mult\cK \le n\mult \<\cK,f \>
  \end{equation}
\end{Lem}
\begin{proof}
  Let $\{h_1,\ldots,h_k\}$ generate the local algebra
  $\cO_p(\cL)/\<\cK,f\>$. Then a simple computation shows that
  $\{h_if^j\}_{1\le i\le k,0\le j\le n-1}$ generate $\cO_p(\cL)/\cK$.
\end{proof}

More generally, let $\cK\subset\cO_p(\C^m)$ be an ideal of finite
multiplicity in the ring of germs at zero of holomorphic functions on
$\C^m$ .
\begin{Lem}
  Assume that $\cK'\supset\cK^n$. Then
  \begin{equation}
    \mult\cK'\le n^m  \mult\cK
  \end{equation}
\end{Lem}
\begin{proof}
  For monomial ideals, the claim is a trivial combinatorial statement.
  To prove claim in the general case, let $\LT(K)$ denote the ideal of
  leading terms of $K$ with respect to an arbitrary monomial ordering. It is well known that
  \begin{equation}
    \mult K = \mult \LT(K),
  \end{equation}
  and that
  \begin{equation}
    \LT(K)^n \subset \LT(K^n),
  \end{equation}
see e.g.  \cite{Eis:book}.
  Thus the general claim follows from the case of monomial ideals.
\end{proof}

\subsection{Poisson extension}

Given a pair $(I,\cI)$ and two functions $F,G\in I$ we define a new
pair $(J,\cJ)$ by letting $J=\<I,\{F,G\}\>$ and $\cJ=\<\cI,\{F,G\}\>$.
Assume that complexity of $F,G$ is known, e.g. they are linear combinations of generators of $I$.
We claim that $(J,\cJ)$ is a controllable extension of $(I,\cI)$.
The complexity of the ideal $J$ can clearly be estimated from that of
$I$.

It is well-known that for a map $\Phi:(\C^n,0)\to (\C^n,0)$ of finite multiplicity the Jacobian $J(f)$ generates a one-dimensional ideal in the local algebra $Q_f$, see \cite{AVG}. As a corollary, we conclude that
the multiplicity of $\cI$ can be estimated from that of $\cJ$.
\begin{Lem} \label{lem:poisson-mult} Let $\cK\subset\cO_p(\cL)$ be an
  ideal of finite multiplicity, and suppose $f,g\in\cK$.  Then
  \begin{equation}
    \mult\cK \le \mult \<\cK, \{f,g\} \>+1
  \end{equation}
\end{Lem}

\subsection{Jacobian extension}
\label{subsec:jacobian-extension}

Let $(I,\cI)$ be a pair and $F\in I$ be of known complexity. Assume that $I$ is a radical
ideal and that all intersections of $I$ and $\cF$ are non-isolated.
Write $F\onL=fh$ where $h$ consists of the factors of $F\onL$ that
vanish on $V(I)\onL$ and $f$ consists of the other factors.  Finally
assume that $f\in\cI$.

Let $h'$ denote the reduced form of $h$. Denote by $k$ the minimal
multiplicity of a factor of $h$. Let $K$ be the maximal multiplicity
of each factor of $h$, so that $h$ divides $(h')^K$. $K$ can be
effectively bounded from above. Indeed, let us take a generic point
$q$ of $h=0$, and let $\ell$ be a generic linear function on $X$
vanishing at $q$. Then $K=\mult_q\<h,\ell\>$, and multiplicity of this
isolated intersection can be bounded from above either using the
considerations above, or (with better bound) using the main result of
\cite{Gab:vf}.

Note that the same upper bound holds for the total number of branches
of $h$, counted with multiplicities: one should take $p$ instead of
$q$.

Given the conditions above, we define a new pair $(J,\cJ)$ by letting
\begin{eqnarray}
  J&=&\<I, \{V_1^{\alpha_1}V_2^{\alpha_2}F \mid \alpha_1+\alpha_2=k\}  \> \\
  \cJ&=&\<\cJ,h'\>.
\end{eqnarray}
We claim that $(J,\cJ)$ is a controllable extension of $(I,\cI)$.

First, to prove that this is a pair it suffices to show that all
derivatives of order $k$ of $F\onL$ belong to $\cJ$.  This follows
from the Leibnitz rule,
\begin{multline}
  \qquad V_1^{\alpha_1}V_2^{\alpha_2}F\onL = V_1^{\alpha_1}V_2^{\alpha_2}(fh) = \\
  = (\text{derivatives of order $\le k-1$ of
    h})\cdot(\cdots)+(\cdots)f \qquad
\end{multline}
since $f\in\cI\subset\cJ$ and derivatives of order $\le k-1$ of $h$
are divisible by $h'\in\cJ$.

It is clear that the complexity of $J$ can be estimated in terms of
the complexity of $I$. It remains to show that the multiplicity of
$\cI$ can be estimated in terms of the multiplicity of $\cJ$. For this
it will suffice to show that for some explicit number $N$, we have
$(h')^N\in\cI$. Since $(h')^K$ is divisible by $h$, it will suffice to
find $N$ such that $h^N\in\cI$.  We in fact have a stronger statement.
Lemma~\ref{lem:extension-lemma} from the Appendix shows that there
exists a global analytic function $H$ such that $V(I)\subset V(H)$ and
$H\onL$ divides $=h^n$ for $n=2^K$. Since $I$ is radical, $H\in I_{\text{an}}$ and certainly
$h^n=H\onL\in I\onL\subset\cI$ ($ I_{\text{an}}$ is the ideal generated by $I$ in the ring of germs at $p$ of functions holomorphic near $p$).

To end this section we remark that crucially, $V(J)\subsetneq V(I)$.
Indeed, the derivatives of order $k$ of $F$ added to $I$ insure that
the lowest order factors of $h$ are not contained in $V(J)$.

\section{Intersection Multiplicities}

In this section we study the intersection multiplicities of the
restrictions of polynomial functions to the leaf $\cL$. In the first
subsection we consider the case of isolated intersection
multiplicities, and in the following subsection we extend the result
to the case of non-isolated intersections.

\subsection{Isolated intersection multiplicities}

Let $I\subset\cO(X)$ be a polynomial ideal. We will be particularly
interested in the case that $I$ is generated by two polynomial
functions $F,G$.

We will say that the intersection of $I$ (or $V(I)$) and $\cF$ is
isolated at a point $p\in\V(I)$ if there exists an open neighborhood
$U$ of $p$ such that $U\cap\V(I)\cap\cL=\{p\}$, where $\cL$ is the
leaf of $\cF$. In this case, we are interested in bounding
$\mult I\onL$.

\begin{Prop} \label{prop:good-poisson} Let $I$ be a radical ideal and
  suppose that it has isolated intersections with $\cF$.  Then there
  exist $F,G\in I$ of bounded complexity such that $\V(\<I,\{F,G\}\>)\subsetneq\V(I)$.
\end{Prop}
\begin{proof}
  We may assume without loss of generality that $\V(I)$ is
  irreducible. Since the condition of having an isolated intersection
  is open, the generic point of $\V(I)$ is an isolated intersection.

  We claim that at a generic point $p$ of $\V(I)$, $T_pV(I)\trans
  \cF$.
  Indeed, otherwise the intersection of these spaces defines a line
  field whose integral trajectories lie in the intersection
  $\V(I)\cap\cL$ where $\cL$ is the leaf containing $p$, in
  contradiction to the assumption that $p$ is an isolated point of
  intersection.

  The tangent space $T_p\V(I)$ is defined by differentials of
  functions in $I$. By transversality, there exist $F,G\in I$ such
  that $\d F\onL,\d G\onL$ are linearly independent at the point
  $p$. One can assume that $F,G$ are linear combinations of generators of $I$. Thus $\{F,G\}$ is non-vanishing at $p$ and the claim is proved.
\end{proof}

\begin{Cor} \label{cor:isolated-inclusion} For any pair $(I,\cI)$
  there exists an controllable inclusion $(I,\cI)\subset(J,\cJ)$ such
  that $J$ is radical and $V(J)$ is the variety of non-isolated intersections between $I$
  and $\cF$.
\end{Cor}
\begin{proof}
  We obtain this inclusion using Proposition~\ref{prop:good-poisson}
  by forming an alternating sequence of controllable inclusions of
  radical and Poisson type. Each pair of inclusions in this sequence
  reduces the dimension of the set of isolated intersections, so after
  at most $n$ steps the sequence stabilizes on the variety of
  non-isolated intersections.
\end{proof}

Suppose now that $I=\<F,G\>$ and $p$ is an isolated intersection point
of $I$ and $\cF$. Consider the pair $(I,I\onL)$. Applying
Corollary~\ref{cor:isolated-inclusion} and
Proposition~\ref{prop:inclusion-mult-bound} we immediately obtain an
upper bound for $\mult_p I\onL$.

\subsection{Non-isolated intersection multiplicities}
\label{subsec:non-isolated-mult}

We now consider the case of non-isolated intersection multiplicities.
Let $I=\<F,G\>$ and suppose that the intersection is not isolated on
$\cL$ at a point $p$.

Near $p$ we may write
\begin{eqnarray}
  F\onL & = & h_f f \\
  G\onL & = & h_g g
\end{eqnarray}
where $h_f,h_g$ are the factors which are common to $F\onL$ and
$G\onL$ (note that they may appear with different multiplicities). We
stress that $h_{f,g},f,g$ are defined only on $\cL$, and the
decomposition of $F$ and $G$ as products of these factors need not
extend outside of $\cL$. We let
\begin{equation}
  \cI = \<f,g\>.
\end{equation}

 \begin{Thm}\label{thm:main}
 The multiplicity $\mult_p\cI$ can be bounded from
 above in terms of degrees of $V_1,V_2,F,G$ and the dimension $n$.
\end{Thm}

We begin in the same manner as in the isolated-intersection case.

Let $(I_0,\cI_0)=(I,\cI)$. We apply
Corollary~\ref{cor:isolated-inclusion} to obtain an controllable
inclusion $(I_0,\cI_0)\subset(I_1,\cI_1)$ such that $I_1$ is radical and  $V(I_1)$ is the
variety of non-isolated intersections between $I$ and $\cF$.

We are now in position to apply the Jacobian extension controllable
inclusion. Indeed, $V(I_1)$ is the locus of non-isolated
intersections, and with $h_1=h_f$ we have:
\begin{itemize}
\item $V(I_1)\onL=\{h_1=0\}$.
\item $F\onL=h_1 f$.
\item $f\in \cI_0\subset\cI_1$.
\end{itemize}
We thus obtain an inclusion $(I_1,\cI_1)\subset(I_2,\cI_2)$, and as
remarked at the end of subsection~\ref{subsec:jacobian-extension},
$\V(I_2)\subsetneq V(I_1)$.

Applying Corollary~\ref{cor:isolated-inclusion} again, we obtain an
controllable inclusion $(I_2,\cI_2)\subset(I_3,\cI_3)$ such that $I_3$ is radical and
$V(I_3)$ is the variety of non-isolated intersections between $I_2$
and $\cF$.

We are now again in position to apply the Jacobian extension
controllable inclusion. Indeed, $V(I_3)$ has only non-isolated
intersections with $\cF$, and letting $h_3$ denote the factors of
$h_f$ which remain in $V(I_3)\onL$ we have:
\begin{itemize}
\item $V(I_3)\onL=\{h_3=0\}$.
\item $F\onL=h_3 f\cdot(h_f/h_3)$.
\item $f\cdot(h_f/h_3)\in \cI_0\subset\cI_1$.
\end{itemize}
We thus obtain an inclusion $(I_3,\cI_3)\subset(I_4,\cI_4)$, and as
remarked at the end of subsection~\ref{subsec:jacobian-extension},
$\V(I_4)\subsetneq V(I_3)$.

Repeating these two alternating types of controllable inclusions and
applying Proposition~\ref{prop:inclusion-sequence} we obtain an
inclusion $(I_0,\cI_0)\subset(I_{2k+1},\cI_{2k+1})$ where at least $k$
factors of $h_f$ have been removed from $V(I_{2k+1})$.  In particular,
for $K$ equal to the number of factors of $h_f$, we know that
$p\not\in V(I_{2K+1})$. Thus, applying
Proposition~\ref{prop:inclusion-mult-bound} we obtain the required
upper bound for $\mult_p\cI_0=\mult_p\cI$.

\section{Appendix}

In this appendix we prove the lemma used in
subsection~\ref{subsec:jacobian-extension} which may be of some
independent interest.

\begin{Lem} \label{lem:extension-lemma} Let $I\subset\cO(X)$ be an
  ideal and $F\in I$, and suppose that every intersection of $I$ and
  $\cF$ is non-isolated. Write $F\onL=fh$ where $h$ consists of the
  factors of $F\onL$ that vanish on $V(I)\onL$ and $f$ consists of the
  other factors. Finally denote by $\mu$ the multiplicity of $h$ at
  $p$.

  Then there exists a function $H\in\cO_{\text{an}}(X)$ such that
  $H\onL$ divides $h^{2^\mu}$ and $V(I)\subset V(H)$.
\end{Lem}

\begin{proof}
  Let $(x,y)$ denote a system of coordinates on $\cL$, and $(z)$
  denote the coordinates parameterizing the leafs, with $p$ being the
  origin and $\cL=\{z=0\}$.  Possibly making a linear change of
  coordinates in $(x,y)$, we may assume that the projection
  $\pi:(x,y,z)\to(x,z)$ restricted to $\{F=0\}$ defines a
  \emph{ramified} covering map in a neighborhood of the origin.

  Fix an annulus $A_x$ around the origin in $(x)$ and some arbitrary
  point $x_0\in A_x$, and a sufficiently small disc $D_y$ in
  $(y)$. Then the fibre of $\pi$ restricted to $\{F=0\}\cap\cL$ is a
  discrete set $B$ parametrizing the branches of $F$ on $\cL$. In case
  $F$ has repeated factors, we view $B$ as a multiset with the
  appropriate multiplicities.  Finally we write $B$ as the disjoint
  union of branches corresponding to $h$ and $f$, $B=B_h\amalg B_f$.

  On $\cL$, one may express the branches of $\{F=0\}$ as ramified
  functions $y_b(x), b\in B$ defined on $A_x$. Furthermore, for $z$ in
  a sufficiently small polydisc $D_z$, these functions extends as
  holomorphic functions $y_b(x,z)$, possibly ramified over a
  ramification locus $\Sigma_z\subset D_z$. We note that since the
  different branches of $F$ remain far apart in $A_x$ (on $\cL$, and
  hence also for sufficiently small $z$), the only ramification in $z$
  occurs when $B$ is a multi-set. In this case, several branches
  corresponding to the same branch on $\cL$ may be permuted by the
  $z$-monodromy.

  We will call a set $S\subset B$ monodromic if it is invariant as a
  multi-set under the monodromy of $A_x$ on $\cL$, in the sense that
  after applying the monodromy, each branch appears with the same
  multiplicity as it originally did. We stress that this is only a
  condition on $\cL$, and does not imply that the corresponding set
  $\{y_b:b\in S\}$ is monodromic for non-zero values of $z$. For
  $z\neq0$, the $A_x$ monodromy may replace one branch by another, as
  long as the two branches correspond to the same branch on $\cL$.

  Assume now that $S$ in monodromic and $S\subset B_h$, and consider
  the function
  \begin{equation}
    F_S(x,y,z) = \prod_{y_b:b\in S} (y-y_b(x,z)).
  \end{equation}
  On $\cL$, the functions $y_b$ actually extend from $A_x$ to the
  punctured disc of the same radius, as the only ramification point
  occurs at the origin. Furthermore, since $S$ is assumed to be
  monodromic on $\cL$, it follows that $F_S\onL$ is univalued and
  holomorphic outside the origin.  Since it is bounded, it extends
  holomorphically at the origin well. Since $S\subset B_h$, we see
  that $F_S\onL$ divides $h$.

  On nearby leafs, the set $S$ need no longer be monodromic even in
  $A_x$, and the function $F_S$ is not necessarily univalued. However,
  for each fixed value of $z$ it is a well-defined function in $A_x$.
  We develop $F_S$ as a Puiseux series in $A_x$ and let $F'_S$ denote
  the holomorphic part obtained by removing all negative and
  fractional terms.  The result is a holomorphic function in
  $(x,y,z)$, possibly ramified over $\Sigma_z$ but unramified in
  $(x,y)$.

  On nearby leafs, the functions $y_b$ may exhibit several
  ramification points (the ramification at the origin on $\cL$ may
  bifurcate into several ramification points). But suppose that for a
  certain nearby leaf $\cL_{z_0}$ the set $S$ happens to be monodromic
  (as a multiset) with respect to the full monodromy in the $(x)$
  variable. Then, arguing just as we did for the leaf $\cL$, we deduce
  that $F_S$ is in fact holomorphic on $\cL_{z_0}$. Hence for such
  leafs we have $F_S\rest{\cL_{z_0}}\equiv F'_S\rest{\cL_{z_0}}$.

  Now define the function $H$ as follows,
  \begin{equation}
    H = \prod_{\text{$S\subset B_h$ monodromic}} F'_S.
  \end{equation}
  We claim that $H$ is in fact an unramified holomorphic
  function. Indeed, we have already seen that each of the factors
  $F'_S$ is holomorphic in $(x,y)$. Furthermore, the monodromy in the
  $z$ variables permutes the set $\{S:\text{$S\subset B_h$
    monodromic}\}$,
  and hence only permutes the factors in the product defining
  $H$. Thus $H$ is univalued, and since it is also bounded near its
  singular locus, it extends holomorphically there as well.

  We claim that $H$ vanishes on the $V(I)$ (at least in a sufficiently
  small neighborhood of the origin). Indeed, consider some fixed value
  of $z_0$ of $z$. By assumption, all intersections of $V(I)$ with
  $\cL_{z_0}$ are non-isolated, so $V(I)\rest{\cL_{z_0}}$ necessarily
  consists of some set $S_{z_0}$ of branches $y_b$. This set, being
  the set of branches of an analytic set on $\cL_{z_0}$, is
  necessarily monodromic. This implies that it is monodromic on $\cL$
  as a multiset as well.

  We claim that for sufficiently small $z_0$, $S_{z_0}\subset B_h$.
  Otherwise $V(I)\rest{\cL_{z_0}}$ contains branches from $B_f$ for
  arbitrarily small valued of $z_0$, and since $V(I)$ is closed we see
  that $V(I)\onL$ contains branches of $f=0$, in contradiciton to the
  assumptions of the lemma.

  To conclude, $S$ is monodromic and $S\subset B_h$, and hence as we
  have seen $F_S\rest{\cL_{z_0}}\equiv F'_S\rest{\cL_{z_0}}$. Since
  $F_S$ vanishes on $V(I)\rest{\cL_{z_0}}$ by definition, and $F'_S$
  is a factor of $H$, we deduce that $H$ vanishes on
  $V(I)\rest{\cL_{z_0}}$. Since this is true for any sufficiently
  small $z_0$, the claim is proved.

  Finally, on $\cL$ we have shown that for monodromic $S\subset B_h$,
  $F'_S\equiv F_S$ and $F_S\onL$ divides $h$. Since $H$ is the product
  over at most $2^\mu$ sets of the factors $F'_S$, we deduce that
  $H\onL$ divides $h^{2^\mu}$, concluding the proof of the lemma.
\end{proof}

\bibliographystyle{plain}

\end{document}